\documentclass[12pt]{amsart}
\usepackage{latexsym,fancyhdr,amssymb,color,amsmath,amsthm,graphicx,listings,comment}
\usepackage[section]{placeins}
\newtheorem{thm}{Theorem} \newtheorem{lemma}{Lemma} \newtheorem{coro}{Corollary}
\let\paragraph\subsection

\title{Green functions of Energized complexes}
\author{Oliver Knill}
\date{October 18, 2020}
\address{Department of Mathematics \\ Harvard University \\ Cambridge, MA, 02138 }


\begin{document}
\maketitle

\section{Results}

\paragraph{}
A function $h: G \to \mathbb{K}$ from a finite abstract simplicial complex $G$ to a ring $\mathbb{K}$ with
conjugation $x^*$ defines $\chi(A)=\sum_{x \in A} h(x)$ and $\omega(G)=$  
$\sum_{x,y\in G, x \cap y \neq \emptyset} h(x)^* h(y)$.
Define $L(x,y) = \chi(W^-(x) \cap W^-(y))$ and $g(x,y) = \omega(x) \omega(y) \chi(W^+(x) \cap W^+(y))$, where
$W^-(x)=\{z \; | \;  z \subset x\}$, $W^+(x) = \{z \; | \;  x \subset z\}$ and 
$\omega(x)=(-1)^{{\rm dim}(x)}$ with ${\rm dim}(x)=|x|-1$.

\paragraph{} The following relation \cite{KnillEnergy2020} 
only requires the addition in $\mathbb{K}$

\begin{thm}
$\chi(G) = \sum_{x,y \in G} g(x,y)$
\label{Theorem 1}
\end{thm} 

\paragraph{}
The next new {\bf quadratic energy relation} links simplex interaction with
multiplication in $\mathbb{K}$. Define $|h|^2=h^* h=N(h)$ in $\mathbb{K}$.

\begin{thm}
$\omega(G) = \sum_{x,y \in G} \omega(x) \omega(y) |g(x,y)|^2$. 
\label{Theorem 2}
\end{thm}

\paragraph{}
The next determinant identity holds if $h$ maps $G$ to a division algebra $\mathbb{K}$
and {\rm det} is the {\bf Dieudonn\'e determinant} \cite{Dieudonne1943}. The geometry 
$G$ can here be a finite set of sets and does not need the simplical complex axiom
stating that $G$ is closed under the operation of taking non-empty finite subsets.

\begin{thm}
${\rm det}(L)={\rm det}(g)=\prod_{x \in G} h(x)$.
\label{Theorem 3}
\end{thm}

\paragraph{}
If $h:G \to \mathbb{K}$ takes values in the {\bf units} $U(\mathbb{K})$ of $\mathbb{K}$, 
like i.e. $\mathbb{Z}_2,U(1),SU(2),\mathbb{S}^7$ of the
division algebras $\mathbb{R},\mathbb{C},\mathbb{H},\mathbb{O}$, 
the {\bf unitary group} $U(H) \cap \mathbb{K}$ of an operator $C^*$-algebra $\mathbb{K} \subset B(\mathcal{H})$ 
for some Hilbert space $\mathcal{H}$ or the 
units in a ring $\mathbb{K}=O_K$ of integers of a number field $K$,
and if $G$ is a simplicial complex, then:

\begin{thm}
If $h(x)^* h(x)=1$ for all $x \in G$, then $g^* = L^{-1}$.
\label{Theorem 4}
\end{thm}

\paragraph{}
For an overview of simplicial complexes and references, see \cite{AmazingWorld}. 
Except for Theorem~(\ref{Theorem 2}), the results were known \cite{KnillEnergy2020,EnergizedSimplicialComplexes}
in special special cases like the topological $h(x)=\omega(x)$, where $\chi(G)=\sum_{x \in G} \omega(x)$ is the
{\bf Euler characteristic} and $\omega(G)=\sum_{x \sim y \in G}$ $\omega(x) \omega(y)$ is the
{\bf Wu characteristic} \cite{Wu1953,Gruenbaum1970}. The pair $(G,\mathbb{K})$ is an example of a 
{\bf ringed space} or a sheaf. For $\mathbb{K}=\mathbb{Z}$ we might think of $h$ as a {\bf divisor} 
and $h(G)$ as its degree, for $\mathbb{K}=\mathbb{C}$ as a quantum mechanical wave and $|\omega(h)|^2 = h^* h$ 
as an amplitude, for $\mathbb{K}=\mathbb{R}^n$, we might
interpret $h$ as a section of a vector bundle or as an embedding of $G$ in $\mathbb{R}^n$ like for example when
doing a geometric realization of $G$. 

\paragraph{}
When taking $\mathbb{K}=\mathcal{G}$ as the ring generated by 
simplicial complexes and $h(x)=X(x)$, the complex generated by $x \in G$,  
we can see $G \in \mathbb{K} \to \exp(G)={\rm det}(L(G)) \in \mathbb{K}$ as an 
{\bf exponential map} because $\exp(G_1+G_2)=\exp(G_1) \exp(G_2)$ 
as addition $G_1+G_2$ is the disjoint union and Theorem~(\ref{Theorem 3}) shows 
that we get from a sum a product. 

\section{Proofs}

\paragraph{}
We reprove Theorem~(\ref{Theorem 1}) algebraically.
The setup in\cite{KnillEnergy2020} was harder, because we did not
start with the explicit expressions for $g(x,y)$ yet.
Let $\{x_1, \dots, x_n\}$ enumerate 
the elements of $V=\bigcup_{x \in G} x$ and $h$ take values in the 
{\bf free algebra} (monoid ring) generated by the variables $x_1,\dots, x_n$. 
If $\mathbb{K}$ is commutative, we can work with the 
{\bf polynomial algebra} $\mathbb{Z}[x_1, \dots x_n]$. 
The algebraic picture is now transparent:

\begin{proof}
Write $h(x)=x$ and have $x$ be the variable associated to the set $x \in G$. 
The matrix entries of $L$ and $g$ are linear expressions:
$$  L(u,v) = \sum_{x \in G, x \subset u \cap v} x                      \; , $$
$$  g(u,v) = \sum_{x \in G, u \cup v \subset x} \omega(u) \omega(v) x  \; . $$
Seen as such, the claim is the algebraic relation
$$  \sum_{x \in G} x  = \sum_{x \in G} [\sum_{u,v \in G, u \cup v \subset x}  \omega(u) \omega(v)] x \; . $$
Because $x$ is a simplex of Euler characteristic $1$, we have
$\sum_{u \subset x} \omega(u) = 1$ and $\sum_{v \subset x} \omega(v) = 1$ so that also 
$$ [\sum_{u,v \in G, u \cup v \subset x}  \omega(u) \omega(v)] = 
   [\sum_{u \in G, u \subset x}  \omega(u)]^2  = 1 \; . $$
\end{proof}

\paragraph{}
Theorem~(\ref{Theorem 2}) can also be seen algebraically. While one needs to distinguish $xy$ and $yx$ 
in the non-commutative case, associativity does not yet factor in because only products of two elements
occur. The Theorem also so holds also for octonions $\mathbb{K}=\mathbb{O}$ or Lie algebras $x*y  = [x,y]$
or if $x y = \langle x,y \rangle$ is considered to be an inner product.

\begin{proof}
When writing the expressions algebraically, $\omega(G) = \sum_{x \cap y \neq \emptyset} x^* y$
is a {\bf generating function} for all intersection relations in the complex $G$. 
Take a pair of sets $x,y$ which do need to be different and look at the expression $x^* y$ on the left. 
On the right, the term $x^* y$ appears if we consider $g(u,v)$ for any pair $u,v \subset x \cap y$. 
We see especially that $x$ and $y$ need also to have a non-empty intersection to the right. 
We have to show
$$ x^* y = \sum_{u,v \subset x \cap y}  \omega(u) \omega(v) x^* y  \; .$$
We get the same term $x^* y$ on the right because
$$ \sum_{u \cup v \subset x \cap y} \omega(u) \omega(v) 
 [\sum_{u \subset x \cap y} \omega(u)] [ \sum_{v \subset x \cap y} \omega(v)] $$
which is $\chi_{top}(x \cap y)^2=1$.
\end{proof}

\paragraph{}
Theorem~(\ref{Theorem 3}) holds more generally for any set $G$ of non-empty sets, where also
the empty set $\emptyset$ (= void) is allowed. Unlike for simplicial complexes, the class of sets of sets has an involution
$x \leftrightarrow x'=V=\bigcup_{x \in G} x  \setminus x$, assigning to $x$ its complement $x' \in V$.
The proof makes use of this duality switching $W^+(x)$ and $W^-(x)$ as to establish linearity of ${\rm det}$ in one variable, 
we need both a proportionality factor $1$ as well as the affinity factor $0$ in each variable. 

\begin{proof}
Because $L^+(u,v) = \sum_{x \in G, x \subset u \cap v} x$ and
$L^-(u,v) = \sum_{x \in G, u \cup v \subset x} x$ are dual to each other in the 
category of sets of sets, we only need to verify the identity for $L=L^-$ or $L^+$.
We can use induction with respect to the number of elements $n$ in $G$ and use that
if we we lift a property for $L^+$ from $(n-1)$ to $n$, we have also shown it 
for $L^-$. For $n=1$, the situation is clear as then $L^+=L^-=[x]$ is a $1 \times 1$ matrix.
In general because the matrix entries of $L$ are linear in each variable,
a Laplace expansion will show in the induction that the determinant is affine $a_k x_k + b_k$ in 
each variable $x_k$. We need then to establish multi-linearity.
The induction assumption is that for any set of $(n-1)$ sets like $\{ x_2, \dots, x_{n} \}$ or
we have ${\rm det}(L^+)={\rm det}(L^-)=\prod x_i$, which is a multi-linear expression in each of the variables.
Lets assume that $x_1$ is a minimal element as a set then $L^-$ has zero column and row entries if its value is zero 
and deleting these rows and columns produces the connection matrix $L^+$ of a set of sets without 
the $x_1$ set in which some entries are changed. Still as a row is zero, the expression
${\rm det}(L^+(x_1))$ is linear $a x_1$ in $x_1$ for some $a$ and not affine $a x_1 +b$. 
To fix the proportionality factor $a$ we use duality and look at $x_1$ in $L^+(x)$ which corresponds
to take a maximal element $x_n$ in $L^-(x)$ which means that it is a minimal element in the dual picture. 
 Given $G$ with $n$ elements, and $x=x_n$ is maximal, we look at ${\rm det}(L^-(x))$. 
In that matrix $L^-$, only the corner entry $L^-_{n,n}$ contains a linear expression in $x$ and $x$ does not appear anywhere else.
A Laplace expansion shows then that the determinant is of the form $x {\rm det}(A) + b$, where $A$ is the
$n-1 \times n-1$ matrix in which the last column and row is deleted and $b$ is some constant. 
Together, these two insights show adding a new set, the determinant is a linear function in the energy $h(x)$
of that set so that ${\rm det}(L) = \prod_{i=1}^{n} x_i$. 
\end{proof}

\paragraph{}
To see Theorem~(\ref{Theorem 4}), we order $G$ so that if $|x|<|y|$, then the set $x$ comes before $y$
in the listing of $G$. Also this theorem needs the simplicial complex assumption for $G$. 

\begin{proof}
With the elements in $G$ ordered according to dimension, the matrix $g^* L$ is 
(i) upper triangular, (ii) contains terms $|x|^2=x^* x$ in the diagonal
and (iii) contains only sums of terms of the form $|y|^2-|z|^2$ in the upper triangular part. 
If all $|x|^2=1$, these three properties (i),(ii),(iii) then show that $g^* L$ 
is the identity matrix. Now to the proof of the three statements: the product
$(g^* L)(x,y) = \sum_{z \in G} g^*(x,z) L(z,y)$ with 
$g^*(x,z) = \sum_{u, x \cup z \subset u} \omega(x) \omega(z) u^*$
and $L(z,y) = \sum_{v, v \subset z \cap y} v$ is
$$  \sum_{z} \omega(z) g^*(x,z) L(z,y) = \sum_z \sum_{u,v, x \cup z \subset u, v \subset z \cap y} \omega(x) \omega(z) u^* v $$ 
which is $0$ if $y \subset x$ and equal to $x^* y = x^* x=|x|^2=1$ if $x=y$ and which is a sum of terms
$\sum_z \sum_{x \subset z \subset y} \omega(x) \omega(z) |z|^2 = 0$ if $x \subset y$. 
The last follows from $\sum_{x \subset z \subset y} \omega(z) = 0$ rephrasing that the 
reduced Euler characteristic $1-\chi_{top}(X)$ of a simplex $X \subset G$ is zero. 
\end{proof}

\paragraph{}
Let us formulate the last step as a lemma

\begin{lemma}
If $X$ is a complete complex with $n$ elements and $Y \subset X$ is a complete sub-complex 
with $0 <m<n$ elements, then the number of odd and even dimensional simplices in $X$ containing $Y$ 
are the same. 
\end{lemma}
\begin{proof}
Assume $X$ is the complex generated by its largest element $x$ and $Y$ is the complex generated
by its largest element $y$.  Define a map $\phi: z \to z \setminus y$ and build the set of sets 
$\phi(X) = \{ \phi(z), z \in X \}$. It is an extended complete simplicial complex containing also 
the void $\emptyset$, a set of dimension $-1$ satisfies $\omega(\emptyset) = (-1)$. 
The f-vector $(f_{-1},f_0,f_1, \dots, f_{n-m-1})$ has the Binomial coefficients $f_k=B(n-m,k+1)$
as components. Because $f(t) = \sum_{k=0}^{n-m} f_{k-1} t^k = (1+t)^{n-m}$ satisfies $f(-1)=0$
for $n>m$, the number of odd and even dimensional simplices are the same. 
\end{proof}

\paragraph{}
For $m=0$, there is one more even dimensional simplex and odd dimensional and $\sum_{y \subset x} \omega(x)=1$
for $m=n$, there is only the simplex $x$ and $\sum_{x \subset y \subset x} \omega(x) = \omega(x)$.
For $X$ is the complex generated by $x$ which is $X=\{ x=\{ 1,2,3 \},\{1,2\},\{1,3\},\{2,3\}, \{1\},\{2\},\{3\} \}$ 
and $Y=\{ \{ y=\{1,2\} \} \}$ then $\{ \{1,2\}, \{1,2,3\} \}$ is the set of sets in $x$ containing $y$. 
It contains one even and one odd dimensional simplex. 
If $Y=\{ y = \{1\} \}$, then $\{ \{1\},\{1,2\},\{1,3\},\{1,2,3\} \}$ contains two even and two odd dimensional
simplices. 

\section{Remarks}

\paragraph{}
Theorem~(\ref{Theorem 3}) justifies to see $g(x,y)$ as {\bf Green function entries} or potential energy 
values between $x$ and $y$. The notation $N(g(x,y))$ for {\rm arithmetic norm} or the real {\bf amplitudes} 
is commonly used if the ring $\mathbb{K}$ is a {\bf number field} or {\bf ring of integers} 
in a number field like $N(a+ib)=a^2+b^2$ in $\mathbb{K}=\mathbb{Z}[i]$. 
In the case if $\mathbb{K}$ is a $C^*$ algebra, then $N(x)=|h(x)|^2$ is the square of the norm
of the operator $h(x)$ which is the {\bf spectral radius} of the self-adjoint operator $x^* x$. 

\paragraph{}
If $\sum_{x \in G} L(x,x)$ denotes the {\bf trace} of $L$ then 
$$  {\rm str}(L) = \sum_{x \in G} \omega(x) L(x,x) $$
is called the {\bf super trace}. With the {\bf checkerboard matrix}
$S(x,y) = \omega(x) \omega(y)$ one can write ${\rm str}(L)= {\rm tr}(SL)$. Our first proof of 
Theorem~(\ref{Theorem 1}) used the following identity in Corollary~(\ref{Corollary 6}) which had been a key 
when proving the energy theorem in the topological case without the Green-Star formula. It actually identified
with the Green function entries $K(x)=\omega(x) g(x,x)$ as {\bf curvature} which add up by Gauss-Bonnet
to Euler characteristic. Theorem~(\ref{Theorem 3}) then identifies this curvature $K(x)$ as the 
{\bf potential} which all the simplices (including $x$) induce on $x$. 
When $h(x)=\omega(x)$ we have seen $\omega(x) g(x,x) = 1-S(x)$ as the reduced Euler characteristic of the unit
sphere $S(x)$ in the Barycentric refinement graph of $x$. 
Even so this curvature or potential energy is an element in the ring  $\mathbb{K}$,
Theorem~(\ref{Theorem 1}) can be interpreted therefore as an {\bf Gauss-Bonnet} formula

\begin{coro}
$\chi(G) = \sum_{x \in G} K(x) = {\rm tr}(S g) = {\rm str}(g)$.
\label{Corollary 6}
\end{coro}

\begin{proof}
We have 
$$  g(u,v) = \sum_{x \in G, u \cup v \subset x} \omega(u) \omega(v) x  \; , $$
so that 
$$  {\rm str}(g) = \sum_{u} \omega(u) g(u,u) = \sum_u \omega(u) \sum_{x \in G, u \subset x} x $$
Comparing the coefficient of the expression $x$ in $\chi(G) = \sum_{x \in G} x$ with the expression $x$
appearing in ${\rm str}(g) = [\sum_{u \subset x} \omega(u)] x$ gives a match because the topological 
Euler characteristic of a simplex $x \in G$ is $1$. 
\end{proof}

\paragraph{}
Corollary \ref{Corollary 6} can be seen as a connection analogue 
of the discrete McKean Singer formula $\chi(G) = {\rm str}(e^{-H t})$ \cite{McKeanSinger,knillmckeansinger}
for the {\bf Hodge Laplacian} $H=(d+d^*)^2 = D^2$, where $d$ are the {\bf incidence matrices} of the simplicial
complex. The self-adjoint Dirac operator $D$ and its square $D^2=H$ act on the same Hilbert space than the matrices $L,g$ 
and produces a symmetry of the non-zero spectrum of even and odd dimensional forms. The kernel of the blocks of
$H=D^2$ are the Betti numbers. We should see $L$ as the {\bf connection analog} of $D$ and $L^* L$ as the analogue
of $H$. We do not have kernels of $L$ and no Hodge theory for $L$. There are some relations although. 
The matrix $(L+g)^*(L+g)=L^*L + g^* g+2$ and  for one-dimensional complexes, $L+g$ is the 
sign-less Hodge Laplacian. 

\begin{coro}
$\omega(G) = {\rm tr}(Sg^*S g))$
\label{Corollary 6}
\end{coro}
\begin{proof}
Define
$$   A(x,y) = (Sg^*S)(x,y) = \omega(x) \omega(y) g^*(x,y) $$ 
so that we can see this as a {\bf Hilbert-Schmidt inner product}
$$ {\rm tr}( S g^* S g) = {\rm tr}(A g) = \sum_{x,y} A(x,y) g(x,y) = \sum_{x,y} \omega(x) \omega(y) |g(x,y)|^2 \; . $$
Now use Theorem~(\ref{Theorem 2}).
\end{proof}

\paragraph{}
We still have a relation for the {\bf cubic Wu characteristic}
$\omega_3(G) = \sum_{x,y,z \in G} h(x)^* h(y) h(z)$, where the sum is over all triples which pairwise
interact. We have $\omega_3(G) = {\rm tr}((S g)^3)$ but this then starts to fail for 
${\rm tr}((S g)^4)$. We still need to investigate more these higher Wu characteristic $\omega_n(x)$ for $n \geq 3$. 
While for $\omega$ we look only at pair interactions, for $\omega_3$ we look at three point interactions
Because the Green function entries $g(x,y)$ can be through of as the interaction energy 
between $x$ and $y$, it is likely that some ``tensor quantity" like
$L^+(x,y,z) = \chi(W^+(x) \cap W^+(y)) + \chi(W^+(x) \cap W^+(z)) + \chi(W^+(y) \cap W^+(z))$ 
will capture the three point interaction of the three simplices $x,y,z$ better. 

\paragraph{}
If $h$ takes values in a real or complex Hilbert space $\mathcal{H}$, one could replace the pairing 
$h(x)^* h(y) \in \mathbb{K}$ with 
some {\bf inner product} $h(x)^* \cdot h(y) = \langle h(x), h(y) \rangle \in \mathbb{C}$. 
If $h$ takes values in unit spheres of a Hilbert space, one gets then close to an {\bf Ising} or 
a {\bf Heisenberg type model} (i.e. \cite{SimonStatMechanics}).
By the classification of real division algebras, this is natural for $\mathbb{R}$ (Ising),
$\mathbb{C}$ (2D Heisenberg) and $\mathbb{H}$ (3D Heisenberg). 
Also the octonion case $\mathbb{O}$ or any linear space with Hilbert space works.
For non-commutative cases like $\mathbb{H},\mathbb{O}$, the determinant becomes the 
Dieudonn\'e determinant which in the non-commutative division algebra case happens to agree with 
the {\bf Study determinant} and in our case $\prod_x |h(x)|$.
If we do not insist on working with determinants, we can have $h(x)$ take
values in the unit sphere of any Hilbert space and still have $g^* L=1$. 
With the dot product as ``multiplication", the right hand side $1$ in $g^* L = 1$ 
does then have real entries $1$ and operator $1$ entries like the matrices $g$ and $L$. 

\paragraph{}
If $A,B \subset G$ are any subsets with $k$ elements, we can look at {\bf minors} 
${\rm det}(g_{A,B})$ which are matrix entries of the {\bf exterior product} $g \wedge g \cdots  \wedge g$. 
The {\bf Fredholm energy} ${\rm det}(1+g^* g)$ is a sum over all possible amplitudes 
$|{\rm det}(g_{A,B})|^2$, where $A,B \subset G$ have the same cardinality.
This is a {\bf generalized Cauchy-Binet} formula \cite{cauchybinet}
$$ {\rm det}(1+F^T G) = \sum_{P} {\rm det}(F_P) {\rm det}(G_P) $$
which holds for all $n \times m$ matrices $F,G$ and also extends to Dieudonn\'e determinants.
We can think of a subset $A \subset G$ with $|A|=k$ as a
{\bf $k$-particle state} and $\chi(A)$ as a sort of momentum and $\omega(A)$ as a sort of kinetic energy.
For two $A,B$ of cardinality $k$, the minor $g(A,B) = {\rm det}(g_{A,B})$ is a matrix entry of
$\wedge_{j=1}^k g$.  The Cauchy-Binet relation $g^2(A,B) = \sum_C g(A,C) g(C,B)$ and
more generally the $n$'th matrix power $g^n(A,B)$ sums over all paths 
$$ g(A,C_1) g(C_1,C_2) \dots g(C_{n-1},B) \; . $$
We mention this to illustrate that there is a {\bf multi-particle interpretation} 
of the set-up. The determinant ${\rm det}(L)$ is then an $n$-particle quantity. 
The additive energy $\chi(G)$ the quadratic energy $\omega(G)$ and the {\bf Fredholm energy}
$\sum_j 1+|\lambda_j|^2$ are now all natural notions.

\paragraph{}
Unrelated to the {\bf intersection calculus} described in Theorems~(\ref{Theorem 1}) to 
(\ref{Theorem 4}) is an {\bf incidence calculus}
defined by {\bf incidence matrices} $d$ defining an {\bf exterior derivative} satisfying $d^2=0$.
The {\bf Dirac matrix} $D=d+d^*$ and the {\bf Hodge Laplacian} $H=(d+d^*)^2$ are like $L,g$
finite matrices of the same size $n \times n$ than $L$ or $g$. When doing a {\bf Lax deformation} of $D$, we
deform the exterior algebra the matrix entries of $d$ become then ring valued. 
The Hodge matrix $H$ is block diagonal with blocks $H_i$ for which ${\rm dim}({\rm ker}(H_i)=b_i$ are still
{\bf Betti numbers} defining the Poincar\'e polynomial $p(t) = \sum_{j=0} b_j t^j$. 
This information uses the topological $h(x)=\omega(x)$ and by Euler-Poincar\'e,
the topological Euler characteristic $\chi(G)=\sum_x \omega(x)$ is the Poincar\'e polynomial evaluated at $t=-1$.
For Wu characteristic, there is also a {\bf quadratic incidence calculus} by defining 
the exterior derivative $dF(x,y) = F(dx,y)-F(x,dy)$
leading to Betti numbers and a {\bf Wu-Poincar\'e polynomial} $q(t)$, where $q(-1)$ is the
Wu characteristic $\omega(G)=\sum_{x \sim y} \omega(x) \omega(y)$. Also the {\bf Wu-Poincar\'e map}
$q: \mathcal{G} \to \mathbb{Z}[t]$ is a ring homomorphism.
Unlike simplicial cohomology associated with $\chi(G)$, the {\bf quadratic incidence cohomology} associated
with $\omega(G)$ is not a homotopy invariant. But it can distinguish the cylinder from the M\"obius strip.
Also here, the Dirac operator can be deformed in a $\mathbb{K}$-valued frame work (for associative 
$\mathbb{K}$ without changing the quadratic cohomology. 

\paragraph{}
For subset $A \subset G$, the sum $\omega(A) = \sum_{x,y \in A} h(x) h(y)$ does in general 
not relate to the Green function entries $g(x,h)$, where $g(x,y)$ is the Green function
of the entire complex $G$. This also was the case for $\chi(A) = \sum_{x \in A} h(x)$
which is in general not the sum over all green function entries of $G$, nor of $A$ (as the
energy theorem requires that $A$ is a simplicial complex). 
For sub-complexes $A \subset G$ we can take the Green functions of the sub-complex and ignore
the outside $G \setminus A$.  With $\overline{A} = \bigcup_{x \in A} W^+(x)$ as some sort of closure,
we tried to see whether $\chi(\overline{A})$ agrees with $\sum_{x,y \in A} g_A(x,y)$ or
$\omega(\overline{A}) = \sum_{x,y \in A} |g_A(x,y)|^2 \omega(x) \omega(y)$ 
but this also does not seem to work. The quantities $\omega(A)$ depends on how $A$ is embedded in $G$. 
There are interaction energies between $A$ and places outside $A$ if $A$ is not a simplicial complex
itself. The boundary is crucial. We know that for a discrete manifold $G$ without boundary in the topological
case $\omega(G) = \chi(G)$ and for a discrete manifold $G$ with boundary $\delta G$ one has 
$\omega(G) = \chi(G) \setminus \chi(\delta G))$. This implies that $\omega(B)= (-1)^d$ for a closed ball $B$
of dimension $d$ and so $\omega(x) = \omega(X) = (-1)^{{\rm dim}(x)}$ if $X$ is the complete simplicial complex 
generated by a simplex $x$. This is the reason why we denoted the Wu characteristic with $\omega$. 

\paragraph{}
If $h$ takes values in $\{-1,1\}$, then $L,g$ are inverses of each other by Theorem~(\ref{Theorem 4}) 
are {\bf real integral quadratic} forms for which the number of negative eigenvalues
agree with the number of negative $h$ values. This follows from the relation 
${\rm det}(L)={\rm det}(g)=\prod_{x \in G} h(x)$
holding for all $\mathbb{C}$-valued $h$ and which when comparing arguments shows that
$\sum_j {\rm arg}( \lambda_j) = \sum_{x \in G} {\rm arg}(h(x))$. More generally:

\begin{coro}
If $\mathbb{K}=\mathbb{R}$ and $h(x) \neq 0$ for all $x$, then the number of elements
in $G$ with $h(x)>0$ agrees with the number of positive eigenvalues of $L$ or $g$.
\end{coro}

\paragraph{}
In the constant case $h(x)=1$, the matrices $L,g$ are
{\bf integral positive definite quadratic forms} $L,g$
which are inverses of each other $L^{-1}=g$ and which have a {\bf symplectic property}
in that they are iso-spectral \cite{CountingMatrix}. The reason for the association is that
symplectic matrices have the
property that the inverse of a matrix has the same eigenvalues than the matrix itself. 
It is known by a {\bf theorem of Kirby} that if $n$ is even and a $n \times n$ matrix has
this spectral symmetry of $\sigma(L) = \sigma(L^{-1}$, then $L$ is conjugated to a 
symplectic matrix $A$ (meaning $A^T J A = J$ with the standard symplectic matrix $J$ 
satisfying $J^2=-1$ and $J^T=J^{-1}=-J$. The spectral property follows from the definition
$A^{-1} = J^T A^T J$.)  In general, since $L,g=L^{-1}$ are self-adjoint, it follows from the spectral
theorem that there is an orthogonal $U$ such that $L^{-1} = U^T L U$. In the symplectic case,
the unitary matrix is $U=J$. Kirbi's observation is just that that if $n$ is even, there is a 
coordinate system in which $U=J$ and that if $n$ is odd we have an eigenvalue $1$, there is
a coordinate system in which the unitary $U$ decomposes into a $(n-1) \times (n-1)$ symplectic
block $J$ and a $1 \times 1$ block $1$.

\paragraph{}
Still in the case $h(x)=1$, the {\bf spectral Zeta function} of $L$
$\zeta(s) = \sum_{j=1}^n \lambda_j^{-s}$ is an entire function in $s$ satisfying
the functional equation $\zeta(a+ib) = \zeta(-a+ib)$. The reason is that there is
not only the symmetry $\zeta(z)=\zeta(-z)$ but also the symmetry $\zeta(z) = \zeta(z^*)$, where
$z^*$ is the complex coordinate.
The same functional equation $\zeta(a+ib) = \zeta(-a+ib)$ 
for the zeta function holds if $h(x)=\omega(x)$ and if $G$ is one-dimensional.
In general, if $h$ is complex valued, the zeta function needs to be defined properly as
it is not clear which branch of the logarithm to use for each $\lambda$. It should then 
be considered for the matrix $L^* L=|L|^2$ or its inverse $g^* g=|g|^2$ which are 
positive definite self-adjoint matrices and so have real eigenvalues. 

\paragraph{}
If $h:G \to \mathbb{K}$ takes values in the units of a ring of integers $\mathcal{O}$
in a number field $\mathbb{K}$, then $g^* g$ is a positive definite integer quadratic form
over $\mathcal{O}$ and $L^* L$ is the inverse of $g^* g$. They are both positive definite
$\mathcal{O}$-valued quadratic forms. We could also take the iso-spectral $gg^*$ rather than 
$g^*g$ but selfadjoint cases like $g+g^*$ or $(L+g)^* (L+g)$ do not have an inverse in general.
For $\mathbb{K}=\mathbb{C}$, and $h(x) \neq 0$, we get positive definite Hermitian forms
$g^* g$ and $L^* L$. There is a unique Hermitian matrix
$A$ such that $e^{-A} = g^* g$ and $e^A = L^* L$. One can get them by finding
the unitary matrix $U$ with diagonal $U^* (g^* g) U = D$ and $U^* (L^* L) U = D^{-1}$
then defining $A=U \log(D) U^*$. Now we can define for $t \in \mathbb{C}$ the one-parameter group $e^{A t}$ of
operators which for $t=1$ gives $L^* L$ and for $t=-1$ gives $g^* g$.

\paragraph{}
If $\lambda_j$ are the eigenvalues of $A=g^* g$, 
the zeta function $\zeta(s) = \sum_{j=1}^n \lambda_j^{-s}$ which can
be rewritten as ${\rm tr}(g^* g)^{s}) = {\rm tr}(L^* L)^{-s})$. It makes sense for all $s \in \mathbb{C}$.
The {\bf Schr\"odinger equation} $i u'= -A u$ has the solution $u(t) = U(t) u(0) = e^{-i A t} u(0) = (g^* g)^{it} u(0)$
so that ${\rm tr}( U(t)) = \zeta(it)$. The zeta function is therefore both interesting for the random
reversible walk $(L^* L)^n$ (when taking integer $n$ and for the unitary Schr\"odinger flow $(g^* g)^{it}$.
We need only that $h(x)$ takes values in some unitary group of an operator algebra,
so that Theorem~\ref{Theorem 4} applies. We have now an action of the complex plane $\mathbb{C}$ which 
leads to a trace interpretation of the zeta function: 

\begin{coro}
For $H=L^* L=e^{A}$ the flow $H^{s}$ is defined for all complex $s \in \mathbb{C}$ 
and $\zeta(s) = {\rm tr}(H^{s})$.
\end{coro}

With classical Laplacians this is not possible. The zeta function of the circle is related to the 
quantum Harmonic oscillator and is the {\bf Riemann
zeta function}. The trace of the evolution in negative time only exists by analytic continuation and one has
to disregard the zero energy. For classical Laplacians $\Delta$ on functions 
or Hodge Laplacians $(d+d^*)^2$ on forms, the heat flow can not be 
evolved backwards due to the existence of harmonic forms leading non-invertibility. 
Also discrete random walks defined by stochastic matrices not be reversed as there
are always zero eigenvalues.

\paragraph{}
We could also define a {\bf non-linear Schr\"odinger flow} as follows. Let $h(t)=u(t)$ define $L(t)$, then
look at the differential equation $u'(t) = i H(t) u(t)$ in which the energy operator $H(t)=L(t) L(t)^*$
is defined by the wave $u(t)$. A discrete version is to start with $u(0)$, then define $u(1) = L^*_0 L_0 u(0)$
then $u(2) = L^*_1 L_1 u(1)$, where always $L_k$ are defined by the functions $u(k): G \to \mathbb{K}$. 
This flow still defines a zeta function
$\zeta(s) = {\rm tr}(H^{-s})$ but now the eigenvalues $\lambda_k(t)$ move with time 
and we might have to analytically continue to define $\zeta(s)$. We have not yet explored that. 
The possibility to attach operators $L$ to a wave $h: G \to \mathbb{C}$ and then 
{\bf let these operators $L$ act on the wave} is an interesting case, where fields $h$ become {\bf quantized} in the
sense that we attach an operator to a field and let this operator propagate the field. 
This is an ingredient of {\bf quantum field theories}. Only that in this combinatorial settings, it
only in involves combinatorics and linear algebra, leading to non-linear ordinary differential equations.
Because the dynamics does not change the norm of the operators or fields, there is a globally defined
dynamics. We still need to investigate this flow and study its long term properties depending on the geometry $G$. 

\paragraph{}
We will elaborate elsewhere more on the arithmetic of complexes $G$ as the current work is heavily motivated by that.
Complexes generate a natural ring $\mathcal{R}$ in which the addition is the disjoint union and the multiplication is
the Cartesian product. There is a natural norm on this Abelian ring $\mathcal{R}$ given in terms of the 
{\bf clique number} $c(G)$ of the
graph complement of the connection graph of $G$ then defining $|G|={\rm min}_{G=A-B}|c(A)+c(B)|$ in the group completion
of the monoid given by disjoint union. This works as $c(A+B)=c(A) + c(B), c(A \times B) = c(A) c(B)$ for 
simplicial complexes. This defines a norm satisfying the Banach algebra property $|G_1 \times G_2| \leq |G_1 \times G_2|$ 
so that we can complete the ring to a {\bf commutative Banach algebra} and with 
a conjugation even to a {\bf $C^*$-algebra} $\mathcal{K}$ extending the Banach algebra of complex numbers
$\mathbb{C}$ we know for our usual arithmetic constructs. Actually, the complex plane is a sub algebra 
generated by $0$-dimensional complexes, leading to a complex scaling multiplication $G \to \lambda G$ for 
complex $\lambda$. So, the base space $G$ is in $\mathbb{K}$ but also the target ring $\mathbb{K}$ 
can be that space. Now take $\mathbb{K}=\mathcal{K}$. 
For example, we can look at $h(x) = X$, where $X$ is the complex generated by the
set $x$. This function defines $\chi: \mathcal{K} \to \mathcal{K}$ given by $\chi(X) = \sum_{x \in X} h(x)$.
The spectral properties of $L$ and $g$ are such that the spectra are the union of spectra under addition 
and the product of the spectra under multiplication. This shows that for every fixed complex number $s$, the 
value $G \to \zeta_{G}(s)$ is a character and so an element in the Gelfand spectrum of the ring $\mathbb{K}$
which by the Gelfand isomorphism is $C(K)$ for some compact topological space $K$ (it is compact because
$\mathbb{K}$ is unital). The zeta map $s \to \zeta_G(s)/n(G) \in K$, where $n(G)=\zeta_G(0)={\rm tr}(L(G))^0)$ is 
for connected finitely generated simplices the number of elements in $G$, 
which extends to a character in $\mathcal{K}$, 
now embeds the complex line in the compact 
space $K$. We don't know whether this {\bf zeta curve} is dense in the Gelfand spectrum $K$. 
We can for example ask whether $n: \mathcal{K} \to \mathbb{C}$ defined by extending cardinality to $\mathcal{G}$ 
or Euler characteristic $\chi_{top}: \mathcal{K} \to \mathbb{C}$ which are known to be a characters 
correspond to points in the spectrum $K$ of $\mathcal{K}$, can be approximated by a zeta curve. 
This is related to the open question whether we can read off the topological Euler characteristic 
$\chi(G)$ from the spectrum of a natural connection Laplacian $L$ like in the topological case
when $h(x)=\omega(x) \in \{-1,1\}$. 

\section{Examples}

\paragraph{}
Lets take the example, where $\mathbb{K}$ is the free algebra generated 
by the variables $x_1,x_2, \dots, x_n$ augmented by conjugated entries $x_k^*$ defining
$|x_k|^2 = x_k^* x_k$ in an enumeration of 
$V=\bigcup_{x \in G} x = \{x_1,x_2,\dots, x_n\}$. 
For $G=K_2 = \{ \{1\},\{2\},\{1,2\} \}=\{x_1,x_2,x_3 \}$ we have 
$\chi(G) = x_1+x_2+x_3$ and
$\omega(G) = x_1^*x_1 + x_2^* x_2  + x_3^* x_3 + x_1^* x_3 + x_3^* x_1 + x_2^* x_3 + x_3^* x_2$. 
The matrices  
$$ L = \left[ \begin{array}{ccc} x_1 & 0 & x_1 \\ 0 & x_2 & x_2 \\ x_1 & x_2 & x_1+x_2+x_3 \\ \end{array} \right],
   g = \left[\begin{array}{ccc}x_1+x_3&x_3&-x_3\\x_3&x_2+x_3&-x_3\\-x_3&-x_3&x_3\\\end{array}\right] \;  $$
multiply to 
$$ g^* L = \left[
                  \begin{array}{ccc}
                   |x_1|^2 & 0 & |x_1|^2-|x_3|^2 \\
                   0 & |x_2|^2 & |x_2|^2-|x_3|^2 \\
                   0 & 0 & |x_3|^2 \\
                  \end{array}
                  \right] \; . $$

\paragraph{}
For the next example $G=\{\{1\},\{2\},\{3\},\{1,2\},\{1,3\},\{2,3\}\}$
lets use variables $G=\{x,y,z,a,b,c\}$. Now,
$$ L =  \left[
\begin{array}{cccccc}
 x & 0 & 0 & x & x & 0 \\
 0 & y & 0 & y & 0 & y \\
 0 & 0 & z & 0 & z & z \\
 x & y & 0 & a+x+y & x & y \\
 x & 0 & z & x & b+x+z & z \\
 0 & y & z & y & z & c+y+z \\
\end{array}
                  \right] \; , $$
$$ g = \left[
\begin{array}{cccccc}
 a+b+x & a & b & -a & -b & 0 \\
 a & a+c+y & c & -a & 0 & -c \\
 b & c & b+c+z & 0 & -b & -c \\
 -a & -a & 0 & a & 0 & 0 \\
 -b & 0 & -b & 0 & b & 0 \\
 0 & -c & -c & 0 & 0 & c \\
\end{array}
                  \right]  \; .  $$
One can check that $\sum_{x,y} g(x,y) =  a + b + c + x + y + z = \chi(G)$. We
have 
$\omega(G) = a b + b a +a c + c a +a x + x a$  $+a y + y z +b c + c a +b x + x b +b z + z b$
$+c y + y c +c z + z c+$ $|a|^2+|b|^2 +|c|^2$ $+ |x|^2+|y|^2+|z|^2$,
the generating function for the intersection relations. We compute
$\sum_{x,y} \omega(x) \omega(y) g(x,y)^2 = |a+b+x|^2+|a+c+y|^2+|b+c+z|^2-a^2-b^2-c^2$ and can
check that this is the same. We have ${\rm det}(L)={\rm det}(g) = abcxyz$. Finally, we see
$$  g L = \left[
\begin{array}{cccccc}
 |x|^2 & 0 & 0 & |x|^2-|a|^2 & |x|^2-|b|^2 & 0 \\
 0 & |y|^2 & 0 & |y|^2-|a|^2 & 0 & |y|^2-|c|^2 \\
 0 & 0 & |z|^2 & 0 & |z|^2-|b|^2 & |z|^2-|c|^2 \\
 0 & 0 & 0 & |a|^2 & 0 & 0 \\
 0 & 0 & 0 & 0 & |b|^2 & 0 \\
 0 & 0 & 0 & 0 & 0 & |c|^2 \\
\end{array}
                   \right]  \; . $$
If all entries have length $1$, we get the identity matrix.

\paragraph{}
Lets look at the example $G=\{\{1\},\{2\},\{1,2,3\}\}$ which is not a simplicial complex.
Denote the energy variables by $G=\{ x,y,z\}$. Now,
$$ L= \left[ \begin{array}{ccc} x & 0 & x \\ 0 & y & y \\ x & y & x+y+z \\ \end{array} \right], 
   g= \left[ \begin{array}{ccc} x+z & z & z \\ z & y+z & z \\ z & z & z \\ \end{array} \right] \; . $$
We have $\omega(G) = x^2+2 x z+y^2+2 y z+z^2$ and
$$\sum_{x,y} \omega(x) \omega(y) g(x,y)^2 = (x+z)^2+(y+z)^2+7 z^2 $$ which are not the same. We need
the simplicial complex structure.
Also the  energy $\chi(G) = x+y+z$ does not agree with $\sum_{x,y \in G} g(x,y) =  x + y + 9 z$ so
that Theorem~(\ref{Theorem 1}) does not hold. 
We have however ${\rm det}(L)= {\rm det}(g) = xyz$. The determinant identity Theorem~(\ref{Theorem 3}) 
holds in general, also if $G$ is not a simplicial complex.

\bibliographystyle{plain}

\begin{thebibliography}{10}

\bibitem{Dieudonne1943}
J.~Dieudonn{\'e}.
\newblock Les determinants sur un corps non commutatif.
\newblock {\em Bulletin de la S.M.F.}, 71:27--45, 1943.

\bibitem{Gruenbaum1970}
B.~Gr{\"u}nbaum.
\newblock Polytopes, graphs, and complexes.
\newblock {\em Bull. Amer. Math. Soc.}, 76:1131--1201, 1970.

\bibitem{knillmckeansinger}
O.~Knill.
\newblock {The McKean-Singer Formula in Graph Theory}.
\newblock {\\}http://arxiv.org/abs/1301.1408, 2012.

\bibitem{cauchybinet}
O.~Knill.
\newblock Cauchy-{B}inet for pseudo-determinants.
\newblock {\em Linear Algebra Appl.}, 459:522--547, 2014.

\bibitem{AmazingWorld}
O.~Knill.
\newblock The amazing world of simplicial complexes.
\newblock {{\\}https://arxiv.org/abs/1804.08211}, 2018.

\bibitem{CountingMatrix}
O.~Knill.
\newblock The counting matrix of a simplicial complex.
\newblock {\\}https://arxiv.org/abs/1907.09092, 2019.

\bibitem{EnergizedSimplicialComplexes}
O.~Knill.
\newblock Energized simplicial complexes.
\newblock https://arxiv.org/abs/1908.06563, 2019.

\bibitem{KnillEnergy2020}
O.~Knill.
\newblock The energy of a simplicial complex.
\newblock {\em Linear Algebra and its Applications}, 600:96--129, 2020.

\bibitem{McKeanSinger}
H.P. McKean and I.M. Singer.
\newblock Curvature and the eigenvalues of the {L}aplacian.
\newblock {\em J. Differential Geometry}, 1(1):43--69, 1967.

\bibitem{SimonStatMechanics}
B.~Simon.
\newblock {\em The statistical mechanics of lattice gases}, volume Volume I.
\newblock Princeton University Press, 1993.

\bibitem{Wu1953}
Wu~W-T.
\newblock Topological invariants of new type of finite polyhedrons.
\newblock {\em Acta Math. Sinica}, 3:261--290, 1953.

\end{thebibliography}

\end{document}